\documentclass{article}

\title{Hyperconvex representations and exponential growth}
\author{A. Sambarino}
\date{}

\usepackage{marvosym}
\usepackage{psfrag}
\usepackage[english]{babel}
\usepackage{mathrsfs}
\usepackage{verbatim}
\usepackage[ansinew]{inputenc}
\usepackage{amsmath,amsthm,amssymb,amscd}
\usepackage[all,cmtip]{xy}
\usepackage{xspace}
\usepackage[dvips]{color}
\usepackage{epsfig}
\usepackage{pb-diagram}
\usepackage{amsfonts}
\usepackage{graphicx}

\newcommand{\R}{\mathbb{R}}

\renewcommand{\P}{\mathbb{P}}

\renewcommand{\/}{\backslash}
\renewcommand{\k}{\kappa}

\newcommand{\vacio}{\emptyset}

\newcommand{\G}{\Gamma}
\renewcommand{\l}{\ell}
\newcommand{\<}{\left<}
\renewcommand{\>}{\right>}
\newcommand{\E}{\Sigma}
\newcommand{\scr}{\mathscr}
\newcommand{\g}{\gamma}
\renewcommand{\a}{\alpha}

\newcommand{\z}{\zeta}
\newcommand{\w}{\widetilde}
\newcommand{\bord}{\partial}
\newcommand{\vo}[1]{\overline{#1}}

\newcommand{\PGL}[1]{\mathrm{PGL}(#1,\mathbb R)}
\newcommand{\mismo}{\circlearrowleft}
\newcommand{\co}{\beta_1}
\newcommand{\vco}{\beta_d}
\newcommand{\ta}{\Theta}
\newcommand{\vect}{\beta}
\newcommand{\cone}{\scr L}

\newcommand{\bus}{\sigma}

\newcommand{\grupo}{\Delta}
\newcommand{\om}{\omega}

\newcommand{\posgen}{\bord^2\scr F}
\renewcommand{\t}{\theta}

\newcommand{\cono}{{\cone_\rho}^*}

\renewcommand{\L}{\Lambda}

\renewcommand{\cal}{\mathcal}

\DeclareMathOperator{\ii}{i}

\DeclareMathOperator{\fix}{fix}

\DeclareMathOperator{\PSL}{PSL}
\DeclareMathOperator{\GL}{GL}

\DeclareMathOperator{\clase}{C}

\DeclareMathOperator{\inter}{int}

\DeclareMathOperator{\pgl}{PGL}

\DeclareMathOperator{\holder}{Holder}

\DeclareMathOperator{\Hitchin}{Hitchin}
\DeclareMathOperator{\tope}{top}

\newtheorem*{teos}{Theorem}
\newtheorem*{teo1}{Theorem A}
\newtheorem*{teoB}{Theorem B}

\newtheorem{teo}{Theorem}[section]
\newtheorem{cor}[teo]{Corollary}
\newtheorem*{cors}{Corollary}

\newtheorem{lema}[teo]{Lemma}

\newtheorem{prop}[teo]{Proposition}
\newtheorem*{prop1}{Proposition}

\theoremstyle{definition}

\newtheorem{defi}[teo]{Definition}

\newtheorem{obs}[teo]{Remark}

\theoremstyle{remark}

\begin{document}
\maketitle

\begin{abstract} Let $G$ be a real algebraic semi-simple Lie group and $\G$ be the fundamental group of a compact negatively curved manifold. In this article we study the limit cone, introduced by Benoist\cite{limite}, and the growth indicator function, introduced by Quint\cite{quint2}, for a class of representations $\rho:\G\to G$ admitting a equivariant map from $\bord\G$ to the Furstenberg boundary of $G$'s symmetric space together with a transversality condition. We then study how these objects vary with the representation.
\end{abstract}

\tableofcontents

\section{Introduction}

Consider a discrete subgroup of isometries $\G$ of a negatively curved space $X.$ The exponential growth rate $$\limsup_{s\to\infty}\frac{\log\#\{\g\in\G:d(o,\g o)\leq s\}}s$$ plays a crucial role in understanding asymptotic properties of the group $\G$: On nice situations this exponential growth rate coincides with the topological entropy of the geodesic flow on $\G\/X$ on its non wandering set and with the Hausdorff dimension of $\G$'s limit set on the visual boundary of $X.$ Let us cite the work of Margulis\cite{margulistesis}, Patterson\cite{patterson}, Sullivan\cite{sullivan}, just to name a few.

An important difference appears when one considers higher rank geometry. Let us briefly recall the work of Benoist\cite{limite} and Quint\cite{quint2}. 

Consider $G$ a connected real semi-simple algebraic group and consider some discrete subgroup $\grupo$ of $G.$ Let $K$ be a maximal compact subgroup of $G$ $\tau$ the Cartan involution on $\frak g$ for which the set $\fix \tau$ is $K$'s Lie algebra, consider $\frak p=\{v\in\frak g: \tau v=-v\}$ and $\frak a$ a maximal abelian subspace contained in $\frak p.$ 

Let $\E$ be the roots of $\frak a$ on $\frak g,$ $\E^+$ a system of positive roots on $\E$ and $\Pi$ the set of simple roots associated to the choice $\E^+.$ Let $\frak a^+$ be a Weyl chamber and $a:G \to \frak a^+$ the Cartan projection. Fix some norm $\|\cdot\|$ on $\frak a$ invariant under the Weyl group. If $\|\ \|$ is euclidean then $\|a(g)\|$ is the Riemannian distance on $X$ ($=G$'s symmetric space) between $g\cdot o=g[K]$ and $o=[K],$ if $\|\ \|$ is not euclidean then $\|a(g)\|$ can be interpreted as $d(o,g\cdot o)$ for some $G$-invariant Finsler metric on $X.$ The exponential growth rate one is interested in is $$h_\grupo^{\|\ \|}:=\limsup_{s\to\infty}\frac{\log \#\{g\in\grupo:\|a(g)\|\leq s\}}s.$$

Nevertheless, $\frak a$ being higher dimensional, one can consider the directions where the points $\{a(g):g\in\grupo\}$ are. Benoist\cite{limite} has shown that the asymptotic cone of $\{a(g):g\in\grupo\},$ i.e. the limit points of sequences $t_na(g_n)$ where $t_n\in\R$ goes to zero and $g_n$ belongs to $\grupo,$ coincides with the closed cone generated by the spectrum $\{\lambda(g) :g\in \grupo\},$ $\lambda:G\to\frak a^+$ being the Jordan projection. One inclusion is trivial since $$\frac {a(g^n)}n\to\lambda(g)$$ when $n\to\infty$ (c.f. Benoist\cite{limite}).

\begin{teo}[Benoist\cite{limite}]\label{teo:benoistlimite} Assume $\grupo$ is Zariski dense in $G,$ then the asymptotic cone generated by $\{a(g):g\in\grupo\}$ coincides with the closed cone generated by $\{\lambda(g):g\in\grupo\}.$ This cone is convex and has non empty interior.
\end{teo}

This cone is the called the \emph{limit cone} of $\grupo$ and denoted $\cone_\grupo.$ Quint\cite{quint2} is then interested in \emph{how many} elements of $\{a(g):g\in \grupo\}$ are in each direction of $\cone_\grupo:$ Given an open cone $\scr C\subset \frak a^+$ consider the exponential growth rate $$h_{\scr C}^{\|\ \|}:= \limsup_{s\to \infty} \frac{\log \#\{g\in \grupo: a(g)\in \scr C\textrm{ with }\|a(g)\|\leq s\}}{s},$$ the \emph{growth indicator function}, introduced by Quint\cite{quint2}, is then the function $\psi_\grupo:\frak a\to\R\cup\{-\infty\}$ defined as $$\psi_\grupo(v):=\|v\|\inf\{h_{\scr C}^{\|\ \|}:\textrm{ $\scr C$ open cone with $v\in\scr C$}\}.$$ One remarks that $\psi_\grupo$ is homogeneous and independent of the norm $\|\ \|$ chosen.

He shows the following theorem:

\begin{teo}[Quint\cite{quint2}] Let $\grupo$ be a Zariski dense discrete subgroup of $G.$ Then $\psi_\grupo$ is concave and upper semi-continuous, the set $$\{v\in\frak a:\psi_\grupo(v)>-\infty\}$$ is the limit cone $\cone_\grupo$ of $\grupo.$ $\psi_\grupo$ is non negative on $\cone_\grupo$ and positive on its interior.
\end{teo}

Quint\cite{quint2} shows that the exponential growth rate for a given norm $\|\ \|$ is then retrieved as $$\sup_{v\in\frak a-\{0\}}\frac{\psi_\grupo(v)}{\|v\|}=h_\grupo^{\|\ \|}.$$

This work consists in deeper study of these objects for hyperconvex representations. This notion has its origin at the work of Labourie\cite{labourie}.

Consider $\G$ a discrete co-compact torsion free isometry group of a negatively curved Hadamard manifold $\w M$ and denote $\scr F=G/P$ where $P$ is a minimal parabolic subgroup of $G.$ The space $\scr F\times\scr F$ has a unique open $G$-orbit denoted $\posgen$

\begin{defi} A representation $\rho:\G\to G$ is \emph{hyperconvex} if it admits a H\"older continuous $\rho$-equivariant map $\z:\bord\G\to\scr F$ such that whenever $x,y\in\bord\G$ are distinct the pair $(\z(x),\z(y))$ belongs to $\posgen.$
\end{defi}

The main example of hyperconvex representation is the following: Consider $\E$ a closed orientable surface of genus $g\geq2$ and say that a representation $\pi_1(\E)\to\PSL(d,\R)$ is \emph{Fuchsian} if it factors as $$\pi_1(\E)\to\PSL(2,\R)\to\PSL(d,\R)$$ where $\PSL(2,\R)\to\PSL(d,\R)$ is the unique  irreducible linear action $\PSL(2,\R)\curvearrowright\R^d$ (modulo conjugation by $\PSL(d,\R)$) and $\pi_1(\E)\to\PSL(2,\R)$ is co-compact. A \emph{Hitchin component} of $\PSL(d,\R)$ is a connected component of $$\hom(\pi_1(\E),\PSL(d,\R))=\{\textrm{morphisms }\rho:\pi_1(\E)\to\PSL(d,\R)\}$$ containing a Fuchsian representation.

\begin{teos}[Labourie\cite{labourie}]A representation in a Hitchin component of $\PSL(d,\R)$ is hyperconvex.
\end{teos}

In this work we begin by showing the following property of the limit cone of a hyperconvex representation:

\begin{prop1}[Corollary \ref{cor:interior}]Consider $\rho:\G\to G$ a Zariski dense hyperconvex representation, then the limit cone $\cone_{\rho(\G)}$ is contained in the interior of the Weyl chamber $\frak a^+.$
\end{prop1}

Recall that $\cone_\grupo$ is by definition closed, so the statement of the last proposition is  stronger than ``$\lambda(\rho\g)$ belongs to the interior of the Weyl chamber for every $\g\in\G$''.

The last proposition together with theorem C in A.S.\cite{quantitative} imply directly the following precise counting result. For $g$ in $\PGL d$ denote $\lambda_1(g)\geq\lambda_2 (g) \cdots\geq \lambda_d(g)$ the logarithm of the modulus of the eigenvalues (counted with multiplicity) of a lift $\w g\in\GL(d,\R)$ of $g,$ with determinant in $\{-1,1\}.$

\begin{cors} Let $\rho:\G\to\PGL d$ be a Zariski dense hyperconvex representation, and fix some $i\in\{1,\ldots,d-1\},$ then there exists some positive $h=h_i$ such that $$hte^{ht}\#\{[\g]\in[\G]:\lambda_i(\rho\g)-\lambda_{i+1}(\rho\g)\leq t\}\to 1$$ when $t\to\infty,$ where $[\g]$ is the conjugacy class of $\g.$
\end{cors}

Concerning the growth indicator function, we show the following theorem inspired in the work of Quint\cite{quints} for Schottky groups of $G.$

\begin{teo1}[Corollary \ref{cor:tangente}] Let $\rho:\G\to G$ be a Zariski dense hyper\-con\-vex representation, then the growth indicator function $\psi_{\rho(\G)}:\frak a\to\R$ is strictly concave, analytic on the interior of $\cone_{\rho(\G)}$ and with vertical tangent on its boundary.
\end{teo1}

Fix some hyperconvex representation $\rho:\G\to G$ and denote $\psi_\rho$ for its growth indicator function. Since $\psi_\rho$ is strictly concave there exists a unique direction in the interior of the limit cone $\tau_{\rho}^{\|\ \|}\in\inter(\cone_\rho)$ such that the supremum of $\psi_\rho/\|\ \|$ is realized, this direction is called \emph{growth direction} of $\rho(\G)$ for the norm $\|\ \|$ (the uniqueness if this direction is also true for any Zariski dense subgroup $\grupo$ of $G$ assuming the norm $\|\ \|$ is euclidean, which we shall not).

By definition the set of points in $\{a(\rho\g):\g\in\G\}$ outside a given open cone containing $\tau_\rho$ has exponential growth rate strictly smaller than $h_\rho^{\|\ \|}.$

In order to prove theorem A we use dual objects associated to $\cone_\rho$ an $\psi_\rho$: if a linear functional $\varphi\in\frak a^*$ verifies $\varphi\geq\psi_\rho$ then $$\|\varphi\|=\sup_{v:\|v\|=1}\varphi(v)\geq \sup_{\|v\|=1}\psi_\rho(v)=h_\rho^{\|\ \|}.$$ One is then led to consider the set $$D_\rho=\{\varphi\in\frak a^*:\varphi\geq\psi_\rho\}.$$ This set is a subset of the dual cone $\cone^*_\rho=\{\varphi\in\frak a^*:\varphi|\cone_\rho\geq0\}$ since $\psi_\rho|\cone_\rho\geq0.$ 

We then relate the set $D_{\rho(\G)}$ with the thermodynamic formalism of the geodesic flow $\phi_t:\G\/ T^1\w M\mismo.$ This idea is already present in the work of Quint\cite{quints}, nevertheless the way to find this relation is different and this method has the advantage of extending to (for example) Hitchin representations of surface groups.







We now briefly explain this relation:

Recall that periodic orbits of the geodesic flow $\phi_t:\G\/T^1\w M$ are in correspondence with conjugation classes $[\g]\in[\G],$ and recall that the pressure of some potential $f:\G\/T^1X\to\R$ is defined as $$P(f)=\sup \{h(\phi_t,m)+\int fdm:m\ \phi_t\textrm{-invariant probability}\}$$ where $h(\phi_t,m)$ is the metric entropy of $\phi_t$ with respect to the measure $m.$ A probability maximizing $P(f)$ is called an equilibrium state of $f.$ The equilibrium state of $f$ is unique provided that $f$ is H\"older continuous.

Following Quint\cite{quint1} and Ledrappier\cite{ledrappier} one finds a $\G$-invariant H\"older continuous function $F_\rho:T^1\w M\to\frak a$ such that $$\int_{[\g]}F_\rho =\lambda( \rho\g).$$

We then show:

\begin{prop1}[Proposition \ref{prop:algo}] Let $\rho:\G\to G$ be a Zariski dense hyperconvex representation, then the set $D_{\rho(\G)}$ is the set of functionals $\varphi\in\frak a^*$ which are non negative on the limit cone such that $P(-\varphi(F_\rho))\leq0.$
\end{prop1}

We then find the following nice dynamical interpretation of the growth indicator. For $\varphi\in\frak a ^*$ denote $m_{\varphi}$ the equilibrium state of $\varphi(F_\rho):\G\/T^1\w M\to\R.$

\begin{cors}[Corollary \ref{cor:tangente}]Consider $\varphi_0\in\frak a^*$ tangent to $\psi_\rho,$ then the direction where $\varphi_0$ and $\psi_\rho$ are tangent is given by the vector $\int F_\rho dm_{\varphi_0}$ and the value of $\psi_\rho$ in this vector is the metric entropy of the geodesic flow for the equilibrium state $m_{\varphi_0}$ $$\psi_\rho(\int F_\rho dm_{\varphi_0})=h(\phi_t,m_{\varphi_0}).$$
\end{cors}

In the last section of this work we study continuity properties of these objects when the representation $\rho$ varies.

Say that $\rho:\G\to \pgl(d,\R)$ is \emph{strictly convex} if it is irreducible and admits two H\"older continuous $\rho$-equivariant maps $\xi:\bord\G\to\P(\R^d)$ and $\eta:\bord\G\to\P({\R^d}^*)$ such that $$\xi(x)\oplus\eta(y)$$ whenever $x$ and $y$ are distinct. 

\begin{prop1}[Proposition \ref{prop:continualambda1}] The functions $$\rho\mapsto \limsup_{s\to\infty} \frac{\log \#\{[\g]\in[\G]:\lambda_1(\rho\g)\leq s\}}s$$ and $$\rho\mapsto \limsup_{s\to\infty} \frac{\log \#\{[\g]\in[\G]:\lambda_1(\rho\g)-\lambda_d(\rho\g)\leq s\}}s$$ are continuous among strictly convex representations.
\end{prop1}

Consider a closed hyperbolic oriented surface $\E$ and denote $\Hitchin(\E,d)$ the Hitchin components of the space $$\hom(\pi_1(\E),\PGL d)/\PGL d.$$ Since Hitchin representations are hyperconvex and irreducible (Labourie\cite{labourie}) they are, in particular, strictly convex.

\begin{cors} The function $:\Hitchin(\E,d)\to\R$ $$\rho\mapsto\lim_{s\to\infty}\frac{ \log\#\{[\g]\in[\G]: (\lambda_1-\lambda_d )(\rho\g)\leq s\}}s$$ is continuous.
\end{cors}

This particular function is shown to be analytic by Pollicott-Sharp\cite{pollicottsharp2}. In fact, what stops us from obtaining more regularity, is that the equivariant map varies (only?) continuously with the representation. This is a consequence of the Anosov property, shown to hold by Guichard-Weinhard\cite{guichardweinhard}.

We return now to Zariski dense hyperconvex representations. We remark that the work of Guichard-Weinhard\cite{guichardweinhard} implies that Zariski dense hyperconvex representations are an open set of the space of all representations $\G\to G.$

We show in corollary \ref{cor:conolimite} that the limit cone varies continuously with the representation. If one fixes a Zariski dense hyperconvex representation $\rho$ and an open cone $\scr C$ contained in the interior of $\cone_\rho,$ it will remain in the interior of the limit cone of all representations nearby. One can thus study the continuity of the growth indicator.

\begin{teoB}[Theorem \ref{teo:continuo}] Let $\rho_0:\G\to G$ be a Zariski dense hyperconvex representation and fix some closed cone $\scr C$ in the interior of the limit cone $\cone_{\rho_0}$ of $\rho_0.$ Consider some neighborhood $U$ of $\rho_0$ such that $\scr C$ is contained in $\inter(\cone_\rho)$ for every $\rho\in U,$ then the function $:U\to\R$ given by $$\rho\mapsto \psi_\rho|\scr C$$ is continuous.
\end{teoB}

We then find the following corollary:

\begin{cors}[Corollary \ref{cor:entropia}] The function that associates to a Zariski dense hyperconvex representation $\rho$ the exponential growth rate $h_{\rho(\G)}^{\|\ \|}$ is continuous.
\end{cors}

\subsubsection*{Acknowledgments}

The author is extremely grateful to Jean-Fran\c cois Quint for useful discussions and guiding.

\section{Anosov flows and H\"older cocycles}

\subsubsection*{Reparametrizations}

Let $X$ be a compact metric space, $\phi_t:X\mismo$ a continuous flow on $X$ without fixed points and $f:X\to\R$ a positive continuous function. Set $\k:X\times\R\to\R$ as \begin{equation}\label{equation:k} \k(x,t)=\int_0^tf\phi_s(x)ds,\end{equation} if $t$ is positive, and $\k(x,t):=-\k(\phi_tx,-t)$ for $t$ negative. Thus, $\k$ verifies the cocycle property $\k(x,t+s)=\k(\phi_t x,s)+\k(x,t)$ for every $t,s\in\R$ and $x\in X.$

Since $f>0$ and $X$ is compact $f$ has a positive minimum and $\k(x,\cdot)$ is an increasing homeomorphism of $\R.$ We then have an inverse $\a:X\times\R\to\R$ that verifies \begin{equation}\label{equation:inversa} \a(x,\k(x,t))=\k(x,\alpha(x,t))=t\end{equation} for every $(x,t)\in X\times\R.$

\begin{defi}\label{defi:repa}The \emph{reparametrization} of $\phi_t$ by $f$ is the flow $\psi_t:X\mismo$ defined as $\psi_t(x):=\phi_{\a(x,t)}(x).$ If $f$ is H\"older continuous we shall say that $\psi_t$ is a H\"older reparametrization of $\phi_t.$
\end{defi}

We say that some function $U:X\to\R$ is $\clase^1$ \emph{in the flow's direction} if for every $p\in X$ the function $t\mapsto U(\phi_t(p))$ is of class $\clase^1$ and the function $$p\mapsto \left.\frac{\partial }{\partial t}\right|_{t=0}U(\phi_t(p))$$ is continuous. Two H\"older potentials $f,g:X\to\R$ are then said to be \emph{Liv\v sic cohomologous} if there exists a continuous $U:X\to\R,$ $\clase^1$ in the flow's direction, such that for all $p\in X$ one has $$f(p)-g(p)=\left.\frac{\partial}{\partial t}\right|_{t=0} U(\phi_t(p)).$$

\begin{obs} When two H\"older potentials $f,g:X\to\R_+^*$ are Liv\v sic cohomologous the reparametrization of $\phi_t$ by $f$ is conjugated to the reparametrization by $g,$ i.e. there exists a homeomorphism $h:X\to X$ such that for all $p\in X$ and $t\in\R$ $$h(\psi_t^fp) =\psi_t^g(hp).$$
\end{obs}

If $m$ is a $\phi_t$-invariant probability on $X$ and $\psi_t$ is the reparametrization of $\phi_t$ by $f,$ then the probability $m'$ defined by $dm'/dm(\cdot)=f(\cdot)/m(f)$ is $\psi_t$-invariant. In particular, if $\tau$ is a periodic orbit of $\phi_t$ then it is also periodic for $\psi_t$ and the new period is $$\int_\tau f=\int_0^{p(\tau)}f(\phi_s(x))ds$$ for where $p(\tau)$ is the period of $\tau$ for $\phi_t$ and $x\in\tau.$ This relation between invariant probabilities induces a bijection and Abramov\cite{abramov} relates the corresponding metric entropies: \begin{equation}\label{eq:abramov}h(\psi_t,m')=h(\phi_t,m)/\int fdm.\end{equation}

Denote $\cal M^{\phi_t}$ the set of $\phi_t$-invariant probabilities. The \emph{pressure} of a continuous function $f:X\to\R$ is defined as $$P(\phi_t,f)=\sup_{m\in\cal M^{\phi_t}}h(\phi_t,m)+\int_X fdm.$$ A probability $m$ such that the supremum is attained is called an \emph{equilibrium state} of $f.$ An equilibrium state for the potential $f\equiv0$ is called a probability with maximal entropy and its entropy is called the topological entropy of $\phi_t,$ denoted $h_{\textrm{top}}(\phi_t).$

\begin{lema}[\S2 of A.S.\cite{quantitative}]\label{lema:entropia2} Consider $\psi_t:X\mismo$ the reparametrization of $\phi_t:X\mismo$ by $f:X\to\R_+^*,$ and assume that $h_{\tope}(\psi_t)$ is finite. Then the bijection $m\mapsto m'$ induces a bijection between equilibrium states of $-h_{\tope}(\psi_t)f$ and probabilities of maximal entropy of $\psi_t.$
\end{lema}


\subsubsection*{Anosov flows}

Assume from now on that $X$ is a compact manifold and that the flow $\phi_t:X\mismo$ is $\clase^1.$ We say that $\phi_t$ is \emph{Anosov} if the tangent bundle of $X$ splits as a sum of three $d\phi_t$-invariant bundles $$ TX=E^s\oplus E^0\oplus E^u,$$ and there exist positive constants $C$ and $c$ such that: $E^0$ is the direction of the flow and for every $t\geq0$ one has: for every $v\in E^s$ $$\|d\phi_tv\|\leq Ce^{-ct}\|v\|,$$ and for every $v\in E^u$ $\|d\phi_{-t}v\|\leq Ce^{-ct}\|v\|.$ 



One can compute the topological entropy of a reparametrization of an Anosov flow as the exponential growth rate of its periodic orbits.

\begin{prop}[Bowen\cite{Bowen1}]\label{prop:bowenentropia} Let $\psi_t:X\mismo$ be a reparametrization of an Anosov flow, then the topological entropy of $\psi_t$ is $$h_{\tope}(\psi_t)=\limsup_{s\to\infty}\frac{\log\#\{\tau\textrm{ periodic}:p(\tau)\leq s\}}s,$$ where $p(\tau)$ is the period of $\tau$ for $\psi_t.$
\end{prop}

As shown by Bowen\cite{bowen2}  transitive Anosov flows admit Markov partitions and thus the ergodic theory of suspension of sub shifts of finite type extends to this flows.

\begin{prop}[Bowen-Ruelle\cite{bowenruelle}]\label{teo:ruellebowen} Let $\phi_t:X\mismo$ be a transitive Anosov flow. Then given a H\"older potential $f:X\to\R$ there exists a unique equilibrium state for $f.$
\end{prop}

\begin{prop}[cf. Ruelle\cite{ruelle}-Ratner\cite{ratnerpresion}]\label{prop:pression} Let $\phi_t:X\mismo$ be a transitive Anosov flow and $f,g:X\to\R$ be H\"older continuous. Then the function $t\mapsto P(f-tg)$ is analytic and $$\left.\frac{\partial P(f-tg)}{\partial t}\right|_{t=0}=-\int gdm_f$$ where $m_f$ is $f$'s equilibrium state. If $\int g dm_f=0$ and $$\left.\frac{\partial^2 P(f-tg)}{\partial t^2}\right|_{t=0}=0$$ then $g$ is cohomologous to zero. Thus, if $g$ is not cohomologically trivial and $\int gdm_f=0$ then $t\mapsto P(f-tg)$ is strictly convex.
\end{prop}

We will need the following lemma of Ledrappier\cite{ledrappier}.

\begin{lema}[Ledrappier\cite{ledrappier}, page 106]\label{lema:lemaledrappier} Consider some potential $f:X\to\R$ such that $\int_\tau f\geq0$ for every periodic orbit $\tau.$ If the number $$h:=\limsup_{s\to\infty} \frac{\log\#\{\tau\textrm{ periodic}: \int_\tau f\leq s\}}s$$ belongs to $(0,\infty)$ then $P(-hf)=0.$ Conversely, if $P(-s_0f)=0$ for some $s_0\in(0,\infty)$ then $$s_0=\limsup_{s\to\infty} \frac{\log\#\{\tau\textrm{ periodic}: \int_\tau f\leq s\}}s=h.$$ If this is the case $$0<\inf_{\tau \textrm{ periodic}} \frac1{p(\tau)}{\int_\tau f}\leq \sup_{\tau \textrm{ periodic}} \frac1{p(\tau)}{\int_\tau f}<\infty.$$
\end{lema}



\subsubsection*{H\"older cocycles on $\bord\G$}

Denote $\G$ for a discrete co-compact torsion free isometry group, of a negatively curved complete simply connected manifold $\w M.$ $\G$ is then a hyperbolic group and its boundary $\bord\G$ is naturally identified with $\w M$'s visual boundary.

We will now focus on H\"older cocycles on $\bord\G.$

\begin{defi}\label{defi:cociclo}A \emph{H\"older cocycle} is a function
$c:\G\times\bord\G\to\R$ such that $$c(\g_0\g_1,x)=c(\g_0,\g_1x)+c(\g_1,x)$$ for
any $\g_0,\g_1\in\G$ and $x\in\bord\G,$ and where $c(\g,\cdot)$ is a H\"older
map for every $\g\in\G$ (the same exponent is assumed for every $\g\in\G$). 
\end{defi}

Given a H\"older cocycle $c$ we define the \emph{periods} of $c$ as the numbers
$$\l_c(\g):=c(\g,\g_+)$$ where $\g_+$ is the attractive fixed point of $\g$ in
$\G-\{e\}.$ The cocycle property implies that the period of an element $\g$ only
depends on its conjugacy class $[\g]\in[\G].$


The main result we shall use on H\"older cocycles is the following theorem of Ledrappier\cite{ledrappier} which relates them to H\"older potentials on $T^1\w M.$

\begin{teo}[Ledrappier\cite{ledrappier}, page 105]\label{teo:ledrappier} For each H\"older cocycle $c$ there exists a H\"older continuous $\G$-invariant function $F_c:T^1\w M\to\R$ such that for every $\g\in\G$ one has $$\l_c(\g)=\int_{[\g]} F_c,$$ where $[\g]$ denotes the periodic orbit of the geodesic flow associated to $\g.$
\end{teo}

Recall we have denoted $\int_{[\g]} F_c$ for the integral of $F_c$ along the periodic orbit associated to $\g.$

One can find an explicit formula for such $F_c$ as follows (Ledrappier\cite{ledrappier} page 105): Fix some point $o\in \w M$ and consider a $\clase^\infty$ function $f:\R\to\R$ with compact support such that $f(0)=1, f'(0)=f''(0)=0$ and $f(t)>1/2$ if $|t|\leq 2\sup\{d(p,\G\cdot o):p\in \w M\}.$

We can assume that $t\mapsto f(d(\phi_t(p,x)_b,q))$ is differentiable on $t$ for every $p,q\in \w M,$ where $\phi_t(p,x)_b\in  \w M$ is the base point of $\phi_t(p,x)\in T^1 \w M.$

Set $A:\w M\times\bord\G\to\R$ to be \begin{equation}\label{eq:formula1}A(p,x)=\sum_{\g\in\G}f(d(p,\g o))e^{-c(\g^{-1},x)},\end{equation} then the function $F_c:\w M\times \bord\G \to\R$ \begin{equation}\label{eq:formula} F_c(p,x)=-\left.\frac{d}{dt}\right|_{t=0} \log A(\phi_t(p,x)_b,x)\end{equation} is $\G$-invariant and verifies $\int_\g F_c=c(\g,\g_+).$

From the explicit formula for $F_c$ one can deduce some regularity properties. 

Denote $\holder^\a(X)$ for the set of H\"older continuous real valued functions $f:X\to\R$ with exponent $\a$ where $X$ is some compact metric space. For $f\in\holder^\alpha(X)$ denote $\|f\|_\infty :=\max |f|$ and $$K_f=\sup \frac{|f(p)-f(q)|}{d(p,q)^\alpha},$$ one then defines the norm $\|f\|_\alpha$ as $\|f\|_\alpha:=\|f\|_\infty+K_f.$ The vector space $(\holder^\alpha(X),\|\ \|_\alpha)$ is a Banach space.

If $c$ is a H\"older cocycle with exponent $\alpha$ and $\g\in\G$ define $\|c(\g,\cdot)\|_\alpha$ as its H\"older norm on $\holder^\a(\bord\G).$

Fix a finite generator $\cal A$ of $\G$ and define the distance between two H\"older cocycles with same H\"older exponent $c$ and $c'$ as $$d(c,c'):= \sup\{\|c(\g,\cdot)-c'(\g,\cdot)\|_\alpha:\g\in\cal A\}.$$ Denote $\cal C^\alpha$ the vector space of H\"older cocycles with exponent $\alpha.$ It is clear that the topology of $\cal C^\alpha$ does not depend on the (finite) generator of $\G.$

\begin{cor}\label{cor:continuidadF} The function $:\cal C^\alpha\to\holder^\a(T^1\w M)$ $$c\mapsto F_c$$ given by formula (\ref{eq:formula}) is analytic.
\end{cor}

\begin{proof} Consider a compact fundamental domain of $\G$ acting on $\w M$ and let $W$ be  a small neighborhood of this compact set. The set $\cal A=\{\g\in\G: \g W\cap W\neq \vacio\}$ is finite and a generator of $\G.$

It is then clear that the function given by formula (\ref{eq:formula1}) $c\mapsto \log A(\cdot,\cdot)|W\times\bord\G$ is analytic since only the elements in $\g\in\cal A$ verify $\g W\cap W\neq\vacio$ and $\cal A$ is finite. An explicit formula for the derivative $$t\mapsto -\left.\frac{d}{dt}\right|_{t=0} \log A(\phi_t(p,x)_b,x)$$ shows analyticity of $c\mapsto F_c|W\times\bord\G.$ Since $F_c$ is $\G$-invariant and $W$ contains a fundamental domain we obtain that $c\mapsto F_c$ is analytic.
\end{proof}

Liv\v sic\cite{livsic}'s theorem implies that the set of $\G$-invariant H\"older functions $F:T^1\w M\to\R$ cohomologous to zero is a closed subspace of $\holder^\a(T^1\w M),$ one obtains thus the following:

\begin{cor}The function $\cal C^\a\to\holder^\a(T^1\w M)/\{\textrm{Liv\v sic cohomology}\}$ $$c\mapsto \textrm{the cohomology class of }F_c,$$ is analytic. 
\end{cor}

We will always assume that the periods of a H\"older cocycle $c$ are positive, i.e. $\l_c(\g)>0$ for every $\g.$ For such a cocycle one defines the exponential growth rate as
$$h_c:=\limsup_{s\to\infty}\frac{\log\#\{[\g]\in[\G]:\l_c(\g)\leq s\}}s\in(0,\infty]$$ (it is consequence of Ledrappier's theorem that $h_c$ is always positive).

\begin{lema}[\S2 of A.S.\cite{quantitative}]\label{lema:funcionpositiva} Let $c:\G\times\bord\G\to\R$ be a H\"older cocycle with finite exponential growth rate, then the function $F_c$ is cohomologous to a positive function and is not cohomologous to a constant.
\end{lema}

Denote $\holder_+^\a(T^1\w M)$ for the subset of $\holder^\a(T^1\w M)$ of functions cohomologous to a positive function.

\begin{lema}\label{lema:entropia} The function $h:\holder^\a_+(T^1\w M)\to\R$ given as a solution to the equation $$P(-h(F)F)=0$$ is analytic. Moreover, the function $F\mapsto \textrm{equilibrium state of $-h(F)F$}$ is also analytic.
\end{lema}

\begin{proof} This is direct consequence of the implicit function theorem and of the formula $$\left.\frac{\partial P(f-tg)}{\partial t}\right|_{t=0}=\int gdm_f$$ where $m_f$ is $f$'s equilibrium state.
\end{proof}

Denote $\cal C^\alpha_+$ the subset of H\"older cocycles with positive periods such that $h_c\in(0,\infty).$ We obtain the following proposition:

\begin{prop}\label{cor:continuidadentropia} The exponential growth rate function $h:\cal C^\alpha_+\to\R$ $$c\mapsto h_c$$ is analytic.
\end{prop}

\begin{proof} Consider some $c\in\cal C^\alpha_+.$ Since $h_c$ is finite and positive lemma \ref{lema:funcionpositiva} implies that the function $F_c$ belongs to $\holder_+^\a(T^1\w M).$ One then applies corollary \ref{cor:continuidadF} together with lemma \ref{lema:entropia}.

\end{proof}

\section{Convex representations}


This section is devoted to the study of the limit cone of convex representations. We first work on strictly convex representations, i.e. irreducible morphisms $\rho:\G\to \PGL d$ admitting equivariant mappings to $\P(\R^d)$ and $\P({\R^d}^*)$ with a transversality condition. We then use these representations to study Zariski dense hyperconvex representations.

\subsubsection*{Strictly convex representations}

Recall that $\G$ is the fundamental group of compact negatively curved manifold. Fix some finite dimensional real vector space $V.$

\begin{defi}We shall say that an irreducible representation $\rho:\G\to\pgl(V)$ is \emph{strictly convex} if there exist two H\"older $\rho$-equivariant mappings $\xi:\bord \G \to\P(V)$ and $\eta:\bord\G\to\P(V^*)$ such that for every distinct points $x\neq y$ on $\bord\G$ the line $\xi(x)$ doesn't belong to the kernel of $\eta(y).$
\end{defi}

We say that $g\in\pgl(V)$ is \emph{proximal} if (any lift of $g$ to $\GL(V)$) has a unique complex eigenvalue of maximal modulus, and its generalized eigenspace is one dimensional. This eigenvalue is necessarily real and its modulus is equal to $\exp\lambda_1(g).$ We will denote $g_+$ the $g$-fixed line of $V$ consisting of eigenvectors of this eigenvalue and denote $g_-$ the $g$-invariant complement of $g_+$ (this is $V=g_+\oplus g_-$). $g_+$ is an attractor on $\P(V)$ for the action of $g$ and $g_-$ is a repelling hyperplane.

\begin{lema}[\S5 of A.S.\cite{quantitative}]\label{lema:loxodromic} Let $\rho:\G\to \pgl(V)$ be a strictly convex representation. Then for every $\g\in\G$ $\rho(\g)$ is proximal, $\xi(\g_+)$ is its attracting fixed line and $\ker \eta(\g_-)$ is the repelling hyperplane, where $\xi$ and $\eta$ are the $\rho$-equivariant maps of the definition. 
\end{lema}

Fix some strictly convex representation $\rho:\G\to\pgl(V).$ The choice of a norm $\|\ \|$ on $V$ induces a H\"older cocycle $\co:\G\times\bord\G\to\R$ defined as $$\co(\g,x)=\log\frac{\|\rho(\g) v\|}{\|v\|}$$ where $v$ belongs to the line $\xi(x).$ We remark that lemma \ref{lema:loxodromic} implies that the period $\co(\g,\g_+)$ is exactly $\lambda_1(\rho\g),$ the logarithm of the spectral radius of $\rho\g.$

The following proposition is key in this work, it states that the cocycle $\co$ has finite  exponential growth rate.

\begin{prop}[\S5 of A.S.\cite{quantitative}]\label{prop:growth} Let $\rho:\G\to\pgl(V)$ be strictly convex, then the cocycle $\co$ has finite exponential growth rate, this is $$\limsup_{s \to\infty} \frac{\log\#\{[\g]\in[\G]:\lambda_1(\rho\g)\leq s\}}s<\infty$$ where $[\g]$ is the conjugacy class of $\g$ in $\G.$
\end{prop}

Let $\frak v$ be the Cartan algebra $$\frak v=\{(w_1,\dots,w_d)\in\R^d:w_1+\ldots+w_d=0\}$$ of $\PGL d.$ We will show that the limit cone of a strictly convex representation doesn't intersect the walls $\{w\in\frak v^+:w_1=w_2\}$ and $\{w\in\frak v^+:w_{d-1}=w_d\}.$ The following lemma is from Benoist\cite{convexes1}.

\begin{lema}[Benoist\cite{convexes1}]\label{lema:lambda2} Let $g\in\pgl(V)$ be proximal and let $V_{\lambda_2(g)}$ be the sum of the  characteristic spaces of $g$ whose eigenvalue is of module $\exp\lambda_2(g).$ Then for every $v\notin\P(g_-)$ with non zero component in $V_{\lambda_2(g)}$ one has $$\lim_{n\to\infty}\frac{\log d_\P(g^n(v),g_+)}n= \lambda_2(g)-\lambda_1(g).$$
\end{lema}

\begin{proof} Consider $u\in g_+$ and $a$ the eigenvalue of $u.$ By definition
one has $\lambda_1(g)=\log|a|.$ We consider then $T:g_-\to\P(\R^d)$ as
$Tw=\R(w+u).$ $T$ identifies the hyperplane $g_-$ to the complement of $\P(g_-)$
in $\P(V).$ The action of $g$ on $\P(V)$ is read, via this identification,
as $\hat g:g_-\to g_-$ $$\hat g(w)=\frac 1agw.$$

One then finds, with a linear algebra argument, that $$\frac 1n \log
\frac{\|g^nw\|}{|a|^n}\to\lambda_2(g)-\lambda_1(g)$$ for every $w\in g_-$ that
is not contained in the characteristic spaces of eigenvalue with module $<
\exp\lambda_2(g).$
\end{proof}

\begin{lema}[cf. Yue\cite{yue}]\label{lema:lambda2hyp} There exist two positive
constants $a$ and $b$ such that for every $\g\in\G$ and any point
$x\in\bord\G-\{\g_-\}$ one has $$-a|\g|\leq\lim_{n\to\infty}\frac{\log
d_o(\g^nx,\g_+)}n\leq -b|\g|.$$
\end{lema}

One obtains the following corollary:

\begin{cor}\label{cor:cono} Let $\rho:\G\to\pgl(d,\R)$ be strictly convex, then there exists $k>0$ such that for any $\g\in\G$ one has $$\frac{\lambda_1 \rho(\g)-\lambda_2\rho(\g)}{\lambda_1(\rho\g)}>k.$$ Consequently the limit cone of $\rho(\G)$ doesn't intersect the walls $\{w\in\frak v^+:w_1=w_2\}$ and $\{w\in\frak v^+:w_{d-1}=w_d\}.$
\end{cor}

\begin{proof} Since $\rho(\g)$ is proximal and its attractive line is
$\xi(\g_+)$ one finds, after lemma \ref{lema:lambda2} and that fact that $\rho$
is irreducible, that $$\lim_{n\to\infty}\frac{\log
d_{\P}(\rho(\g)^n\xi(x),\xi(\g_+))}n=\lambda_2\rho(\g)-\lambda_1\rho(\g)$$ for a
point $x\in\bord\G-\{\g_-\}.$ The fact that $\xi$ is H\"older then implies that 

$$\lambda_2\rho(\g)-\lambda_1\rho(\g)\leq\lim_{n\to\infty}\frac 1n \log Cd(\g^n
x,\g_+)^\kappa$$ $$=\lim_{n\to\infty}\kappa\frac 1n \log d_o(\g^n x,\g_+).$$
Lemma \ref{lema:lambda2hyp} implies that this quantity is smaller than $-\kappa
b|\g|.$ In order to finish the proof we need to compare $|\g|$ with $\lambda_1(\rho\g)$ for wich we apply proposition \ref{prop:growth} together with Ledrappier\cite{ledrappier}'s lemma \ref{lema:lemaledrappier} to the cocycle $\co:$ $$0<\frac 1m<\inf_{[\g]} \frac{\lambda_1 (\rho(\g))} {|\g|}\leq \sup_{[\g]}\frac{\lambda_1(\rho(\g))}{|\g|}< m$$ for some constant $m>1.$
\end{proof}

\begin{prop}[Proposition 4.10 of Guichard-Weinhard\cite{guichardweinhard} + proposition 2.1 of Labourie\cite{labourie}]\label{prop:GW1} The function that associates to a strictly convex representation its equivariant maps is continuous and the H\"older exponent of the equivariant maps can be chosen locally constant.
\end{prop}

We are now able to prove the following proposition. Recall that $\lambda_d(g)$ is the logarithm of the modulus of $g$'s smallest eigenvalue.

\begin{prop}\label{prop:continualambda1} The functions $$\rho\mapsto h_1(\rho):= \limsup_{s\to\infty} \frac{\log \#\{[\g]\in[\G]:\lambda_1(\rho\g)\leq s\}}s$$ and $$\rho\mapsto h_{1d}(\rho):= \limsup_{s\to\infty} \frac{\log \#\{[\g]\in[\G]:\lambda_1(\rho\g)-\lambda_d(\rho\g)\leq s\}}s$$ are continuous among strictly convex representations.
\end{prop}

\begin{proof}Since the cocycle $\co$ has finite exponential growth rate, proposition \ref{prop:GW1} together with corollary \ref{cor:continuidadentropia} imply directly the continuity of $h_1(\rho).$ 

We focus then on $h_{1d}(\rho).$ The dual representation $\rho^*:\G\to\pgl({\R^d}^*)$ given by $\rho^*(\g)\varphi=\varphi \circ\rho(\g^{-1})$ is also strictly convex. The cocycle associated to $\rho^*,$ $$\vco(\g,x)=\log\frac{\|\rho^*(\g)\varphi\|}{\|\varphi\|}$$ where $\varphi\in\eta(x),$ has periods $$\vco(\g,\g_+)=\lambda_1(\rho\g^{-1})=-\lambda_d(\rho\g).$$ Consider now the H\"older cocycle $\beta_{1d}:\G\times\bord\G\to\R$ defined as $$\beta_{1d}(\g,x)= \co(\g,x)+\vco(\g,x).$$ The periods of $\beta_{1d}$ are $$\beta_{1d}(\g,\g_+)=\lambda_1(\rho\g)+\lambda_1(\rho\g^{-1})=\lambda_1(\rho\g)-\lambda_d(\rho\g)>0$$ for every $\g\in\G.$ Again by proposition \ref{prop:GW1} and corollary \ref{cor:continuidadentropia} it is sufficient to prove that the cocycle $\beta_{1d}$ has finite exponential growth rate, but this is clear from the inequality $$\lambda_1(g)-\lambda_d(g)\geq\lambda_1(g)$$ for every $g\in\pgl(d,\R)$ together with the fact that $h_1(\rho)$ is finite. This finishes the proof.
\end{proof}

\subsubsection*{Convex representations on some flag space}\label{theta}

Strictly convex representations are then used to study Zariski dense representations $:\G\to G$ which have equivariant maps to $G/P$ where $P$ is some parabolic subgroup of $G$:

Consider a real semi-simple algebraic group $G.$ Let $K$ be a maximal compact subgroup of $G$ $\tau$ the Cartan involution on $\frak g$ for wich the set $\fix \tau$ is $K$'s Lie agebra, consider $\frak p=\{v\in\frak g: \tau v=-v\}$ and $\frak a$ a maximal abelian subspace contained in $\frak p.$ 

Let $\E$ be the roots of $\frak a$ on $\frak g,$ $\E^+$ a system of positive roots on $\E$ and $\Pi$ the set of simple roots associated to the choice $\E^+.$ To each subset $\t$ of $\Pi$ one associates a parabolic subgroup $P_\t$ of $G$ whose Lie algebra is, by definition, $$\frak p_\t=\frak a \oplus\bigoplus_{\a\in\E^+}\frak g_\a\oplus \bigoplus_{\a\in \<\Pi-\t\>}\frak g_{-\a}$$ where $\<\t\>$ is the set of positive roots generated by $\t$ and $$\frak g_\a=\{w\in\frak g:[v,w]=\a(v)w\ \forall v\in\frak a\}.$$ Every parabolic subgroup of $G$ is conjugated to a unique $P_\t.$

Set $W$ to be the Weyl group of $\E$ and note $u_0:\frak a\to \frak a$ the biggest element in $W,$ $u_0$ is the unique element in $W$ that sends $\frak a^+$ to $-\frak a^+.$  The \emph{opposition involution} $\ii:\frak a\to\frak a$ is the defined as $\ii:=-u_0.$

Fix from now on $\t\subset\Pi$ a subset of simple roots of $G$ and write $\scr F_\t$ for $G/P_\t.$  We consider also $P_{\ii(\t)},$ the parabolic group associated to $$\ii(\t):=\{\a\circ\ii:\a\in\t\}.$$ The set $\scr F_{\ii(\t)}\times\scr F_\t$ possesses a unique open $G$-orbit, which we will denote $\posgen_\t.$

\begin{defi}We shall say that a representation $\rho:\G\to G$ is $\t$-\emph{convex} if there exist two $\rho$-equivariant H\"older maps $\xi:\bord\G\to\scr F_\t$ and $\eta:\bord\G\to\scr F_{\ii\t}$ such that if $x\neq y$ are distinct points in $\bord\G$ then the pair $(\eta(y),\xi(x))$ belongs to $\posgen_\t.$
\end{defi}

A $\Pi$-convex representation (i.e., when the set $\t$ is the full set of
simple roots and thus the parabolic group $P_\Pi$ is minimal) is called \emph{hyperconvex}. The set $\scr F_\Pi$ is the Furstenberg boundary $\scr F$ of $G$'s symmetric space.

Consider $\{\om_\a\}_{\a\in\Pi}$ the set of fundamental weights of $\Pi.$ 
 
\begin{prop}[Tits\cite{tits}]\label{prop:titss} For each $\alpha\in\Pi$ there
exists a finite dimensional proximal irreducible representation
$\L_\alpha:G\to\pgl(V_\alpha)$ such that the highest weight $\chi_\alpha$ of
$\L_\alpha$ is an integer multiple of the fundamental weight $\om_\alpha.$ 
\end{prop}

Fix some $\t$ and consider some $\a\in\t.$ Consider also $\L_\a:G\to\PGL(V_\a)$ a representation given by Tits's proposition. Since $\L_\a$ is proximal and $\a\in\t,$ one obtains an equivariant mapping $\xi_\a:\scr F_\t\to\P(V_\a).$

The highest weight of the dual representation $\L_\a^*:G\to\P(V_\a^*)$ is $\chi_\a\ii,$ one thus obtains an equivariant mapping $\eta_\a:\scr F_{\ii\t}\to\P(V_\a^*).$ Moreover, the pair $(x,y)\in\posgen_\t$ verifies $$\eta_\a(x)|\xi_a(y)\neq0.$$ One deduces the following remark:

\begin{obs}\label{obs:remark} If $\rho:\G\to G$ is Zariski dense and $\t$-convex then the composition $\L_\a \circ\rho:\G\to\pgl(V_\a)$ (where $\L_\a$ is Tits's representation for $\a\in\t$) is strictly convex.
\end{obs}

Remark \ref{obs:remark} together with corollary \ref{cor:cono} imply the following corollary:

\begin{cor}\label{cor:interior} Let $\rho:\G\to G$ be a Zariski dense $\t$-convex representation. Then the limit cone of $\cone_{\rho(\G)}$ of $\rho(\G)$ does not intersect the walls $\{v\in\frak a: \a(v)=0\}$ for every $\a\in\t\cup\ii\t.$ 

In particular, the limit cone of a Zariski dense hyperconvex representation is contained in the interior of the Weyl chamber $\frak a^+.$
\end{cor}

\begin{proof} 



As observed before, if $\a\in\t$ the composition $\L_\a\rho:\G\to\pgl(V_\a)$ is strictly convex. Applying corollary \ref{cor:cono} for the representation $\L_\a\rho$ implies the existence of some $\kappa_\a>0$ such that $$\frac {\a(\lambda(\rho\g))}{\chi_\a (\lambda (\rho\g)) }=\frac {\lambda_1 (\L_\a\rho\g) -\lambda_2(\L_a\rho\g)}{\lambda_1 (\L_\a \rho\g) }>\kappa_\a.$$ 
\end{proof}

\subsubsection*{Busemann cocycle}

We shall now focus on hyperconvex representations, i.e. $\rho:\G\to G$ admits a H\"older continuous equivariant map $\z:\bord\G\to\scr F$ such that the pair $(\z(x),z(y))$ belongs to $\posgen$ whenever $x\neq y.$

Given such a representation there is a natural H\"older (vector) cocycle on the boundary of $\G$ that appears for which we need \emph{Buseman's cocycle} on $G$ introduced by Quint\cite{quint1}: The set $\scr F$ is $K$-homogeneous with stabilizer $M,$ where $K$ is a maximal compact subgroup of $G,$ one then defines $\bus:G\times\scr F\to\frak a$ to verify the following equation $$gk=l\exp(\bus(g,kM))n$$ following Iwasawa's decomposition of $G=Ke^{\frak a} N,$ where $N$ is the unipotent radical of $P_\Pi=P.$

The cocycle one naturally associates to a hyperconvex representation is then $\vect^\rho:\G\times\bord\G\to\frak a$ defined as $$\vect^\rho(\g,x)= \bus(\rho(\g), \z(x)).$$ Let $\lambda:G\to\frak a ^+$ be the Jordan projection. If there is no confusion we will omit the superscript of $\vect^\rho=\vect.$

\begin{lema}[\S5 of A.S.\cite{quantitative}]\label{lema:espectro} The periods of $\vect$ are $\vect(\g,\g_+)=\lambda(\rho\g).$
\end{lema}




We will consider linear functionals on the dual cone $${\cone_\grupo}^*:= \{\varphi\in\frak a^*:\varphi|\cone_\grupo\geq0\}.$$

\begin{lema}[\S5 of A.S.\cite{quantitative}]\label{lema:finite} Let $\rho:\G\to G$ be a Zariski dense hyperconvex representation and consider some $\varphi$ in the dual cone $\cone_\rho^*,$ then the H\"older cocycle $\varphi\circ\vect:\G\times\bord\G\to\R$ has finite exponential growth rate if and only if $\varphi$ belongs to the interior of $\cone_\rho^*.$
\end{lema}



For a hyperconvex representation $\rho:\G\to G$ consider $F_\rho:T^1\w M\to\frak a$ given by Ledrappier\cite{ledrappier}'s theorem \ref{teo:ledrappier} for the cocycle $\vect:\G\times\bord\G\to\frak a.$ The following corollary is direct consequence of the last lemma and lemma \ref{lema:funcionpositiva}.

\begin{cor}\label{cor:positiva2} Let $\rho:\G\to G$ be a Zariski dense hyperconvex representation and fix some $\varphi$ in the interior of the dual cone ${\cone_\rho}^*.$ Then the function $\varphi(F_\rho):T^1\w M\to\R$ is cohomologous to a positive function, and is not cohomologous to a constant.
\end{cor}





\section{The growth indicator is strictly concave}

We shall now consider the \emph{growth indicator function} introduced by Quint\cite{quint2}. Recall that $a:G\to\frak a^+$ is the Cartan projection and fix some norm $\|\ \|$ on $\frak a$ in variant under the Weyl group.

Consider $\grupo$ a discrete Zariski dense subgroup of $G.$ For an open cone $\scr C$ on $\frak a^+$ consider the exponential growth rate $$h_{\scr C}:=\limsup_{s\to\infty}\frac{\#\{g\in\grupo:a(g)\in\scr C\textrm{ and }\|a(g)\|\leq s\}}s.$$One then sets $\psi_\grupo:\frak a^+\to\R$ as $$\psi_\grupo(v):=\|v\|\inf_{\scr C\textrm{ open cone}:v\in\scr C} h_{\scr C}.$$Remark that $\psi_\grupo$ is homogeneous and does not depend on the norm chosen.

\begin{teo}[Quint\cite{quint2}]\label{teo:teoquint} The function $\psi_\grupo$ is concave and upper semicontinuous, positive on $\cone_\grupo$ and strictly positive on its relative interior. The set $\{v\in\frak a:\psi_\grupo(v)>-\infty\}$ coincides with the limit cone $\cone_\grupo.$
\end{teo}

We need the following lemma of Quint\cite{quint2}:

\begin{lema}[Lemma 3.1.3 of Quint\cite{quint2}]\label{lema:algunacosa} Let $\grupo$ be a Zariski dense subgroup of $G$ and consider $\varphi\in\frak a^*.$ If $\varphi(v)>\psi_\grupo(v)$ for every $v\in\frak a-\{0\}$ then the Poincare series $$\sum_{g\in\grupo}e^{-\varphi(a(g))}<\infty.$$ If there exists $v$ such that $\varphi(v)<\psi_\grupo(v)$ then $$\sum_{g\in\grupo}e^{-\varphi(a(g))}=\infty.$$
\end{lema}

Fix a Zariski dense hyperconvex representation $\rho:\G\to G$ and denote $\psi_\rho$ for its growth indicator function. If $\varphi\in\frak a^*$ verifies $\varphi\geq\psi_\rho$ then $$\|\varphi\|\geq\sup \frac{\psi_\rho(v)}{\|v\|}=h_\grupo.$$ One is then interested on the set $$D_\rho:= \{\varphi\in\frak a^*:\varphi\geq\psi_\rho\}.$$ Since $\psi_\rho$ is non negative on the limit cone $\cone_\rho$ the set $D_\rho$ is contained in the dual cone ${\cone_\rho}^*.$ For $\varphi\in{\cone_\rho}^*$ define $$h_\varphi=\lim_{s\to\infty}\frac{ \log\#\{ \g\in\G :\varphi(a(\rho\g))\leq s\}}s.$$Remark that $h_\varphi$ is the critical exponent of the Poincare series $$\sum_{\g\in\G}:e^{-\varphi(a(\rho\g))}.$$ 

Quint's lemma \ref{lema:algunacosa} implies the following characterization of the set $D_\rho.$

\begin{lema}\label{obs:menor1} The interior of the set $D_\rho$ is the set of $\varphi\in{\cone_\rho}^*$ such that $h_\varphi<1$ and its boundary coincides with the set of linear functionals such that $h_\varphi=1.$
\end{lema}

We are now interested in showing that the growth indicator function $\psi_\rho$ of $\rho(\G)$ is strictly concave with vertical tangent on the boundary of the limit cone $\cone_\rho.$ 

One (trivial) consequence of theorem C on A.S.\cite{quantitative} is the following corollary:

\begin{cor}\label{cor:exponente} If $\varphi\in\cono$ then the exponential growth rate of the H\"older cocycle $\varphi\circ\vect$ coincides with the exponential growth of $\{a(\rho\g):\g\in\G\}$ i.e. $$ \limsup_{s\to\infty} \frac{\log \#\{[\g] \in[\G]: \varphi(\lambda(\rho\g))\leq s\}}s$$ $$=\limsup_{s\to\infty} \frac{\log\#\{\g \in\G:\varphi(a(\rho\g))\leq s\}}s=h_\varphi$$
\end{cor}

This corollary allows us to link the growth indicator function with the thermodynamic formalism on the geodesic flow on $\G\/ T^1\w M.$ We will give a description of the set $D_\rho$ by considering the pressure function of potentials on $\G\/T^1\w M.$ Fix from now on a $\G$-invariant function $F_\rho:T^1\w M\to\frak a$ given by Ledrappier\cite{ledrappier}'s theorem \ref{teo:ledrappier} for the vector cocycle $\vect:\G\times\bord\G\to\frak a.$

\begin{prop}\label{prop:D} Let $\rho:\G\to G$ be a Zariski dense hyperconvex representation. Then $D_\rho=\{\varphi\in\frak a^*:P(-\varphi\circ F_\rho)\leq0\}.$ The interior of $D_\rho$ is then the set $$\{\varphi\in\frak a^*:P(-\varphi\circ F_\rho)<0\}.$$
\end{prop}

\begin{proof} Recall that, after corollary \ref{cor:exponente}, for every $\varphi\in\cono$ one has $$h_\varphi=\limsup_{s\to\infty}\frac{\log\#\{\g\in\G:\varphi(a(\rho\g))\leq s\}}s$$ $$=\limsup_{s\to\infty}\frac{\log\#\{[\g]\in[\G]:\varphi(\lambda(\rho\g))\leq s\}}s.$$ Recall also that after lemma \ref{obs:menor1} the interior of $D_\rho$ is the set of linear functionals $\varphi\in\inter\cono$ with $0<h_\varphi<1.$

Consider then a linear functional $\varphi$ in the interior of $D_\rho,$ this
is, $\varphi(v)>\psi_\rho(v)\ \forall v\in\frak a-\{0\}.$ We want to show that
$P(-\varphi\circ F_\rho)<0.$

Quint\cite{quint2}'s theorem \ref{teo:teoquint} states that $\psi_\rho$ is positive on the interior of the limit cone and thus $\varphi$ belongs to the interior of the
dual cone $\cono,$ this is $\varphi|\cone_\rho-\{0\}>0.$ Moreover one has
$h_\varphi<1.$ 

Corollary \ref{cor:positiva2} implies that $\varphi\circ F_\rho$ is cohomologous to a strictly positive function and is not cohomologous to a constant, proposition \ref{prop:pression} then implies that $t\mapsto P(-t\varphi(F_\rho))$ has strictly negative derivative. Thus $$t\mapsto P(-t\varphi(F_\rho))$$ is strictly decreasing. Ledrappier\cite{ledrappier}'s lemma \ref{lema:lemaledrappier} implies then $P(-h_\varphi\varphi(F_\rho))=0.$ One then finds that $P(-\varphi( F_\rho))<0$ (see figure 1).

\begin{figure}[h]
\begin{center}
\includegraphics{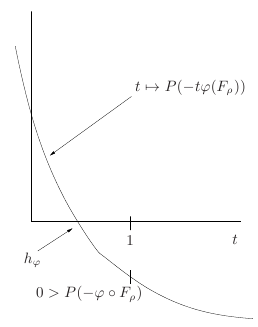}
\caption{The function $t\mapsto P(-t\varphi(F_\rho))$ when $\varphi\circ F_\rho$ is
cohomologous to a positive function and not cohomologous to a constant.} 
\end{center}
\end{figure}

Conversely, fix a linear functional $\varphi\in\frak a^*$ such that $P(-\varphi( F_\rho))<0.$ Considering again the function $t\mapsto P(-t\varphi(F_\rho))$ one finds that, since $P(0)=h_{\tope}(\phi_t)>0$ and $P(-\varphi(F_\rho))<0,$ there exists some $0<h<1$ with $P(-h\varphi(F_\rho))=0.$

By definition of pressure one has $$P(-\varphi(F_\rho))=\sup_{m\in\cal M^{\phi_t}}h(m,\phi_t)-\int \varphi(F_\rho)dm<0,$$ which implies that for every $\g\in\G$ $$\int_{[\g]}\varphi(F_\rho)>0,$$ this is $\varphi\circ \vect$ has positive periods. we can thus apply Ledrappier\cite{ledrappier}'s lemma \ref{lema:lemaledrappier} and conclude that  that such $h$ is necessarily $h_\varphi$ and thus $\varphi$ belongs to the interior of $D_\rho.$
\end{proof}

We will now deduce properties for $\psi_\rho$ from properties of the pressure function. We need Benoist\cite{benoist2}'s theorem below: 

\begin{teo}[Benoist\cite{benoist2}]\label{teo:densos} Consider $\grupo$ a Zariski dense subgroup of $G,$ then the group generated by $\{\lambda(g):g\in \grupo\}$ is dense in $\frak a.$
\end{teo}

\begin{prop}\label{prop:algo} Let $\rho:\G\to G$ be a Zariski dense hyperconvex representation, then the set $D_\rho$ is strictly convex and its boundary is an analytic sub manifold of $\frak a^*.$
\end{prop}

\begin{proof} Fix $F_\rho:T^1\w M\to\frak a$ given by Ledrappier\cite{ledrappier}'s theorem \ref{teo:ledrappier} and consider the function $\vo P:\frak a^*\to\R$ given by $\vo
P(\varphi)=P(-\varphi(F_\rho)).$ Proposition \ref{prop:pression} implies that this
function is analytic and its derivative $d\vo P:\frak a^*\to\frak a$ is given by
the formula $$d \vo P(\varphi)=-\int F_\rho dm_\varphi$$ where $m_\varphi$ is the equilibrium state of $\varphi(F_\rho).$

If $\varphi\in\frak a^*$ is such that $\vo P(\varphi)=0$ then proposition \ref{prop:D} implies that $\varphi$ belongs to the boundary of the set $D_\rho,$ in particular $\varphi\in \cone^*_\rho$ and $h_\varphi=1.$ One deduces that $\varphi(F_\rho)$ is cohomologous to a positive function (corollary \ref{cor:positiva2}) and thus $$\int \varphi(F_\rho)dm_\varphi\neq0.$$ Hence the vector $$d\vo P(\varphi)=\int F_\rho dm_\varphi\neq 0.$$

We conclude that 0 is a regular value of $\vo P$ and thus $\bord D_\rho=\vo P^{-1}\{0\}$ is an analytic sub manifold of $\frak a^*.$

The tangent space to $\bord D_\rho$ at $\varphi_0\in \bord D_\rho$ is $$T_{\varphi_0}\bord D_\rho=\{\varphi\in\frak a^*: \int \varphi(F_\rho) dm_{\varphi_0}=0\}.$$ Consider then $\varphi\in T_{\varphi_0}\bord D_\rho.$ Since the periods of $F_\rho$ generate a dense subgroup of $\frak a$ (Benoist's theorem \ref{teo:densos}) the function $\varphi(F_\rho)$ is not cohomologous to zero. Proposition \ref{prop:pression} then implies that the function $$t\mapsto \vo P(\varphi_0-t\varphi)$$ is strictly convex with one critical point at $0.$ Thus $\varphi_0+T_{\varphi_0}\bord D_\rho$ does not intersect $\bord D_\rho$ (except at $\varphi_0$) and thus $D_\rho$ is strictly convex.
\end{proof}

The following lemma is a consequence of Hahn-Banach's theorem, one can find a
proof in \S 4.1 of Quint\cite{quints}:

\begin{lema}[Duality lemma]\label{lema:cono} Let $V$ be a finite dimensional
vector space and $\Psi:V\to\R\cup\{-\infty\}$ a concave homogeneous upper
semi-continuous function. Set $$V^*_\Psi=\{\Phi\in V^*:\Phi\geq\Psi\}\textrm{
and }L_\Psi=\{x\in V:\Psi(x)>-\infty\}.$$ Suppose that $V^*_\Psi$ and $L_\Psi$
have non empty interior, then 
\begin{itemize}
\item[-] For every $x\in L_\Psi$ one has $$\Psi(x)=\inf_{\Phi\in
V^*_\Psi}\Phi(x),$$ \item[-] the set $V^*_\Psi$ is strictly convex if and only
if $\Psi$ is differentiable on the interior of $L_\Psi$ and with vertical
tangent on the boundary. \item[-] the boundary $\bord V^*_\Psi$ is
differentiable if and only if the function $\Psi$ is strictly concave.
\end{itemize}
When this conditions are satisfied, the derivative induces bijection between the
set of directions contained in the interior of $L_\Psi$ and $\bord V^*_\Psi.$
\end{lema}

We find the following corollary:

\begin{cor}\label{cor:tangente} The growth indicator $\psi_\rho$ of a Zariski dense hypercon\-vex representation $\rho$ is strictly concave, analytic on the interior of the limit cone and with vertical tangent on its boundary. If $P(-\varphi_0(F_\rho))=0$ then $\varphi_0$ is tangent to $\psi_\rho$ in the direction given by the vector $$\int F_\rho dm_{\varphi_0}$$ and the value $$\psi_\rho(\int F_\rho dm_{\varphi_0})=h(\phi_t,m_{\varphi_0})$$ is the metric entropy of the geodesic flow for the equilibrium state $m_{\varphi_0}.$
\end{cor}

\begin{proof} Proposition \ref{prop:algo} together with the duality lemma \ref{lema:cono} imply then that:

\begin{itemize}
\item[i)]Since $D_\rho$ is strictly convex, $\psi_\rho$ is of class $\clase^1$
on the interior of the cone of $\cone_\rho$ but with vertical tangent on its boundary, \item[ii)] Since $\bord D_\rho$ is of class $\clase^1,$
$\psi_\rho$ is strictly concave.
\end{itemize}

Proposition \ref{prop:pression} implies that the derivative of the bijection between interior directions of $\cone_\rho$ and $\bord D_\rho$ is invertible. Thus, since $\bord D_\rho$ is analytic, the derivative of $\psi_\rho$ is analytic on the interior of $\cone_\rho$ and thus $\psi_\rho$ is analytic.

The formula $$d \vo P(\varphi)=-\int F_\rho dm_\varphi$$ together with the first point of the duality lemma \ref{lema:cono} imply that $\varphi_0\in\bord D_\rho$ is tangent to $\psi_\rho$ in the direction given by the vector $\int F_\rho dm_{\varphi_0},$ moreover one has $$\varphi_0( \int F_\rho dm_{\varphi_0})=\psi_\rho( \int F_\rho dm_{\varphi_0}).$$ Since $\varphi_0(F_\rho )$ is cohomologous to a positive function (corollary \ref{cor:positiva2}), we can consider $\sigma_t:\G\/T^1\w M\mismo$ the reparametrization of the geodesic flow by $\varphi_0(F_\rho).$ Applying lemma \ref{lema:entropia2}  and Abramov\cite{abramov}'s formula (\ref{eq:abramov}) we have that the topological entropy of $\sigma_t$ verifies $$h_{\tope}(\sigma_t)=h(\phi_t,m_{\varphi_0})/\int \varphi_0(F_\rho)dm_{\varphi_0}.$$ Following proposition \ref{prop:bowenentropia} the topological entropy of $\sigma_t$ is the exponential growth rate of its periodic orbits, i.e. $h_{\tope}(\sigma_t)=h_{\varphi_0},$ this last quantity is equal to 1 since $\varphi_0\in\bord D_\rho$ and thus $$\psi_\rho (\int F_\rho d m_{\varphi_0}) =\int \varphi_0(F_\rho) dm_{\varphi_0} =h(\phi_t ,m_{\varphi_0}).$$ This finishes the proof

\end{proof}

\section{Continuity properties}

In this section we are interested in showing that the objects one associates to a Zariski dense $\t$-representation vary continuously with the representation. We are concerned in the cone $\cone_\rho,$ the growth indicator $\psi_\rho$ and the growth form $\ta_\rho.$

For a Zariski dense hyperconvex representation $\rho:\bord\G\to G$ denote $F_\rho:T^1\w M\to\frak a$ the function given by Ledrappier\cite{ledrappier}'s theorem \ref{teo:ledrappier} for the cocycle $\vect^\rho$ (this choice is only valid modulo Livsic cohomology).

Recall that $\phi_t:\G\/T^1\w M\mismo$ is the geodesic flow and denote $\cal M^{\phi_t}$ the set of $\phi_t$-invariant probabilities (for all $t$).

\begin{lema}\label{lema:genera} The set $$\{\int F_\rho dm:m\in\cal M^{\phi_t}\}$$ is compact, doesn't contain $\{0\}$ and generates the cone $\cone_\rho^\t.$
\end{lema}

\begin{proof}
Compactness is immediate since $\cal M^{\phi_t}$ is compact. Considering some $\varphi$ in the interior of the dual cone $\cone_\rho^*$ and applying lemma \ref{lema:finite} together with Ledrappier's lemma \ref{lema:lemaledrappier} one obtains that zero does not belong to $\{\int F_\rho dm:m\in\cal M^{\phi_t}\}$. 

The limit cone $\cone_\rho$ is the closed cone generated by the spectrum $\lambda(\rho\g)=\int_{[\g]} F_\rho.$ Since convex combination of periodic orbits are dense in $\cal M^{\phi_t}$ (Anosov's closing lemma c.f.Shub\cite{Shub}) the last statement follows.
\end{proof}

Denote $\hom^Z_\Pi(\G,G)$ for the set of Zariski dense hyperconvex representations endowed with the topology as subset of $G^{\cal A}$ where $\cal A$ is a finite generator of $\G.$

Guichard-Weinhard\cite{guichardweinhard} have shown that Zariski dense hyperconvex representations are the so called Anosov representations, and are thus an open set in the space of representations. From this result together with Labourie\cite{labourie} one obtains the following proposition:

\begin{prop}[Theorem 4.11 of Guichard-Weinhard\cite{guichardweinhard} + proposition 2.1 of Labourie\cite{labourie}]\label{prop:GW} The function that associates to a Zariski dense hyperconvex representation its equivariant maps is continuous and the H\"older exponent of the equivariant maps can be chosen locally constant.
\end{prop}

This proposition together with corollary \ref{cor:continuidadF} give:

\begin{prop}\label{prop:funcion} The function that associates to every $\rho\in\hom^Z_\Pi(\G,G)$ the cohomology class of $F_\rho:T^1\w M\to\frak a$ is continuous.
\end{prop}

One directly obtains the continuity of the cone $\cone_\rho.$

\begin{cor}\label{cor:conolimite} The function $\hom^Z_\Pi(\G,G)\to\{\textrm{compact subsets of } \P(\R^d)\}$ given by $\rho\mapsto\P(\cone_\rho)$ continuous.
\end{cor}

\begin{proof} Obvious from lemma \ref{lema:genera} and proposition \ref{prop:funcion}.
\end{proof}

\begin{cor} Fix some Zariski dense hyperconvex representation $\rho_0$ and consider some $\varphi$ in the interior of the dual cone ${\cone_{\rho_0}}^*.$ Then the function $$\rho\mapsto h_\varphi(\rho):=\lim_{s\to\infty}\frac{\log\#\{[\g]\in[\G]:\varphi(\lambda(\rho\g))\leq s\}}s$$ is continuous in a neighborhood $U$ of $\rho_0$ such that $\varphi|\cone_\rho-\{0\}>0$ for $\rho\in U.$
\end{cor}

\begin{proof} Since the equivariant maps vary continuously with the representation, the function $\rho\mapsto \varphi\circ\vect^\rho:\G\times\bord\G\to\R$ is continuous. The exponential growth rate of $\varphi\circ\vect^\rho$ is $h_\varphi(\rho)$ and thus the corollary is consequence of corollary \ref{cor:continuidadentropia}.
\end{proof}

We can now show theorem B.

\begin{teo}\label{teo:continuo} Let $\rho_0:\G\to G$ be a Zariski dense hyperconvex representation and fix some open cone $\scr C$ such that its closure is contained in the interior of the limit cone $\cone_{\rho_0}$ of $\rho_0.$ Consider a neighborhood $U$ of $\rho_0$ such that $\scr C$ is contained in $\inter(\cone_\rho)$ for every $\rho\in U,$ then the function $:U\to\R$ given by $$\rho\mapsto \psi_\rho|\scr C$$ is continuous.
\end{teo}

\begin{proof} Consider the set $D_{\rho_0}$ of linear functionals $\varphi\in\frak a^*$ such that $\varphi\geq\psi_{\rho_0},$ and more precisely its boundary $\bord D_\rho.$ It suffices to prove that for some fixed $\varphi\in\bord D_{\rho_0}$ and a given neighborhood $W$ of $\varphi$ there exists a neighborhood $U$ of $\rho_0$ such that for every $\rho\in U$ one has that $\bord D_\rho$ intersects $W.$

Proposition \ref{prop:D} states that $\bord D_\rho$ is the set of linear functionals in ${\cone_\rho}^*$ such that $P(-\varphi (F_\rho))=0$ lemma \ref{lema:entropia} together with proposition \ref{prop:funcion} imply thus the result.

\begin{figure}[h]
\begin{center}
\includegraphics{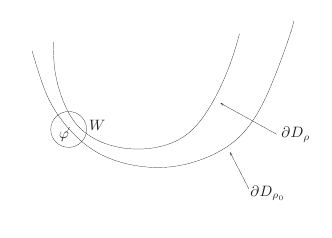}
\caption{The set $D_\rho$ after perturbation.} 
\end{center}
\end{figure}

\end{proof}

Strict concavity of $\psi_\rho$ together with the last theorem imply the following corollaries. Fix some norm $\|\ \|$ on $\frak a$ invariant under the Weyl group.

\begin{cor} The function that $\ta_{\cdot}^{\|\ \|}:\hom^Z_\Pi(\G,G)\to\frak a^*$ that associates to $\rho$ the growth form $\ta_\rho^{\|\ \|}$ is continuous.
\end{cor}

\begin{proof} Recall that $\ta_\rho^{\|\ \|}$ is the functional tangent to $\psi_\rho$ in the direction that $\psi_\rho(\cdot)/\|\cdot\|$ attains it maximum. Since $\psi_\rho$ is strictly concave, this direction belongs to the interior of $\cone_\rho.$ This remarks together with theorem \ref{teo:continuo} imply the result.
\end{proof}


Recall we have denoted $h^{\|\ \|}_\rho$ for the exponential growth rate  $$h_\rho^{\|\  \|}=\lim_{s\to\infty}\frac{\log\#\{\g\in\G:\|a(\rho\g)\|\leq s\}}s.$$

\begin{cor}\label{cor:entropia}  The function $:\hom^\Pi(\G,G)\to\R$ $$\rho\mapsto h_{\rho(\G)}^{\|\ \|}$$ is continuous.
\end{cor}

\begin{proof}Obvious since $h_{\rho}^{\|\ \|}$ coincides with the norm of the growth form, $$\|\ta_\rho^{\|\ \|}\|=h_{\rho(\G)}^{\|\ \|},$$ the result then follows from the last corollary.
\end{proof}

\bibliography{stage}

\begin{thebibliography}{10}

\bibitem{abramov}
L.M. Abramov.
\newblock On the entropy of a flow.
\newblock {\em Dokl. Akad. Nauk. SSSR}, 128, 1959.

\bibitem{limite}
Y.~Benoist.
\newblock Propri{\'e}t{\'e}s asymptotiques des groupes lin{\'e}aires.
\newblock {\em Geom. funct. anal.}, 7(1), 1997.

\bibitem{benoist2}
Y.~Benoist.
\newblock Propri{\'e}t{\'e}s asymptotiques des groupes lin{\'e}aires ii.
\newblock {\em Adv. Stud. Pute Math.}, 26, 2000.

\bibitem{convexes1}
Y.~Benoist.
\newblock Convexes divisibles $\textrm{I}$.
\newblock In {\em Algebraic groups and arithmetic}, pages 339--374. Tata Inst.
  Fund. Res., 2004.

\bibitem{Bowen1}
R.~Bowen.
\newblock Periodic orbits of hyperbolic flows.
\newblock {\em Amer. J. Math.}, 94, 1972.

\bibitem{bowen2}
R.~Bowen.
\newblock Symbolic dynamics for hyperbolic flows.
\newblock {\em Amer. J. Math.}, 95, 1973.

\bibitem{bowenruelle}
R.~Bowen and D.~Ruelle.
\newblock The ergodic theory of axiom $\textrm{A}$ flows.
\newblock {\em Invent. Math.}, 29, 1975.

\bibitem{guichardweinhard}
O.~Guichard and A.~Weinhard.
\newblock Anosov representations: domains of discontinuity and applications.
\newblock ArXiv: 1108.0733v1, 2011.

\bibitem{labourie}
F.~Labourie.
\newblock Anosov flows, surface groups and curves in projective space.
\newblock {\em Invent. Math.}, 165, 2006.

\bibitem{ledrappier}
F.~Ledrappier.
\newblock Structure au bord des vari{\'e}t{\'e}s {\`a} courbure n{\'e}gative.
\newblock {\em S{\'e}minaire de th{\'e}orie spectrale et g{\'e}om{\'e}trie de
  Grenoble}, 71, 1994-1995.

\bibitem{livsic}
A.O. Lopes and Ph. Thieullen.
\newblock Sub-actions for $\tex{A}$nosov flows.
\newblock {\em Ergod. Th. \& Dynam. Sys.}, 25, 2005.

\bibitem{margulistesis}
G.~Margulis.
\newblock Applications of ergodic theory to the investigation of manifolds with
  negative curvature.
\newblock {\em Functional Anal. Appl.}, 3, 1969.

\bibitem{patterson}
S.-J. Patterson.
\newblock The limit set of a fuchsian group.
\newblock {\em Acta mathematica}, 136, 1976.

\bibitem{pollicottsharp2}
M.~Pollicott and R.~Sharp.
\newblock Length asymptotics in higher teichm\"uller theory.
\newblock Preprint.

\bibitem{quint2}
J.-F. Quint.
\newblock Divergence exponentielle des sous-groupes discrets en rang
  sup{\'e}rieur.
\newblock {\em Comment. Math. Helv.}, 77, 2002.

\bibitem{quint1}
J.-F. Quint.
\newblock Mesures de $\textrm{P}$atterson-$\textrm{S}$ullivan en rang
  sup{\'e}rieur.
\newblock {\em Geom. funct. anal.}, 12, 2002.

\bibitem{quints}
J.-F. Quint.
\newblock L'indicateur de croissance des groupes de $\textrm{S}$chottky.
\newblock {\em Ergod. Th. \& Dynam. Sys.}, 23, 2003.

\bibitem{ratnerpresion}
M.~Ratner.
\newblock The central limit theorem for geodesic flows on $n$-manifolds of
  negative curvature.
\newblock {\em Israel J. Math.}, 16, 1973.

\bibitem{ruelle}
D.~Ruelle.
\newblock {\em Thermodynamic $\textrm{F}$ormalism}.
\newblock Addison-Wesley, London, 1978.

\bibitem{quantitative}
A.~Sambarino.
\newblock Quantitative properties of convex representations.
\newblock Submited. arXiv:1104.4705v1, 2011.

\bibitem{Shub}
M.~Shub.
\newblock {\em Global stability of dynamical systems}.
\newblock Springer Verlag, 1987.

\bibitem{sullivan}
D.~Sullivan.
\newblock The densitiy at infinity of a discrete group of hyperbolic motions.
\newblock {\em Publ. Math. de l'I.H.E.S.}, 50, 1979.

\bibitem{tits}
J.~Tits.
\newblock Repr{\'e}sentations lin{\'e}aires irr{\'e}ductib les d'un groupe
  r{\'e}ductif sur un corps quelconqe.
\newblock {\em J. Reine Angew. Math.}, 247, 1971.

\bibitem{yue}
C.~Yue.
\newblock The ergodic theory of discrete isometry groups on manifolds of
  variable negative curvature.
\newblock {\em Trans. of the A.M.S.}, 348(12), 1996.

\end{thebibliography}
\bibliographystyle{plain}

\author{$\ $ \\
Andr\'es Sambarino\\
  Laboratoire de Math\'ematiques\\
Universit\'e Paris Sud,\\
  F-91405 Orsay France,\\
  \texttt{andres.sambarino@math.u-psud.fr}}

\end{document}